\documentclass[letterpaper, 10 pt, conference]{ieeeconf}  

\IEEEoverridecommandlockouts                              
\newcommand{\HF}{H_{F_{i_k}^{k}}}                                          \newcommand{\Hf}{H_{f,U_{i_k}}}
\newcommand{\Hp}{H_{\psi_{i_{k}}}}  
\newcommand{\Hpi}{H_{\psi_{i}}} 
\newcommand{\xki}{x_{k,i_{k}}}
\newcommand{\xkip}{x_{k,i_{k}-1}}     
\newcommand{\alphaki}{\alpha_{i_{k}}}  
\newcommand{\dkik}{d_{k}^{(i_{k})}}  
\newcommand{\xkik}{x_{k}^{(i_{k})}}
\newcommand{\xkikn}{x_{k+1}^{(i_{k})}}  
\newcommand{\xbarik}{\bar{x}_{k}^{(i_{k})}}  
\newcommand{\HfW}{H_{f,W}}
\newcommand{\HfV}{H_{f,V}} 
\newcommand{\Lik}{L_{i_{k}}^{k}} 
\newcommand{\Hpsimax}{L_{\psi,\max}}    

\usepackage{graphicx}
\usepackage{psfrag}
\usepackage{subfigure}
\usepackage{url}
\usepackage{color}
\usepackage{cite}
\usepackage{epsfig}
\usepackage{multirow}
\usepackage{amsmath,amssymb}
\usepackage{algorithmic}
\usepackage{algorithm}
\usepackage{amsfonts,mathrsfs}
\usepackage{verbatim}
\usepackage{acronym}
\usepackage{array}
\usepackage[utf8]{inputenc}
\usepackage[english]{babel}
\usepackage[mathscr]{euscript}
\DeclareMathOperator*{\argmin}{arg\,min}
\newtheorem{definition}{Definition}

\newtheorem{theorem}{Theorem}
\newtheorem{assumption}[theorem]{Assumption}
\newtheorem{lemma}{Lemma}

\title{\LARGE \bf  Coordinate projected  gradient descent  minimization and its application to orthogonal nonnegative matrix factorization }

\author{Flavia Chorobura$^{1}$, Daniela Lupu$^{1}$, Ion Necoara$^{1}$
	\thanks{$^{1}$Automatic Control and System Engineering Department, University Politehnica  Bucharest, Spl. Independentei, 060042 Bucharest, Romania, Emails: {\tt\small flavia.chorobura@stud.acs.upb.ro, daniela.lupu@upb.ro, ion.necoara@upb.ro.}}%
}

\begin{document}

	\maketitle
	\thispagestyle{empty}
	\pagestyle{empty}

	\begin{abstract}
		In this paper we consider large-scale composite nonconvex  optimization problems having  the  objective function formed as a sum of three  terms, first has block coordinate-wise Lipschitz continuous gradient, second  is twice differentiable but nonseparable and third is the indicator function of some separable closed convex set.  Under these general settings we derive and analyze a new cyclic coordinate descent method, which uses the partial gradient of the differentiable part of the  objective, yielding a  coordinate gradient descent scheme with a novel adaptive stepsize rule. We prove that this stepsize rule makes the coordinate gradient   scheme a descent method, provided that additional assumptions hold for the second term in the objective function.  We also  present a worst-case complexity analysis for this new method in the nonconvex settings.  Numerical results on orthogonal nonnegative matrix factorization problem also  confirm the  efficiency of our  algorithm.  
	\end{abstract}


	\section{INTRODUCTION}
	\noindent In this paper we study a cyclic (block) coordinate gradient  descent method for solving the  large-scale  composite optimization problem: 
	\begin{equation}
		\label{eq:prob}
		F^* = \min_{x \in \mathbb{R}^{n}} F(x) := f(x) + \psi(x) + \phi(x) ,
	\end{equation}
	where $f: \mathbb{R}^n  \to \mathbb{R}$  has block coordinate-wise Lipschitz gradient, $\psi : \mathbb{R}^n  \to {\mathbb{R}}$ is twice differentiable  (both functions possibly nonseparable and  nonconvex),  and $\phi : \mathbb{R}^n  \to {\mathbb{R}}$ is the indicator function of a convex closed separable set $Q =  \Pi_{i = 1}^{n} Q_{i}$.  Optimization problems having this composite structure arise in many applications such as   orthogonal nonnegative matrix factorization \cite{AhoHie:21} and distributed control  \cite{NecCli:13}.   When the dimension of these problems is large,  the usual methods based on full gradient and Hessian  perform poorly.  Hence, it is reasonable  to solve such large-scale problems using  (block) coordinate descent methods. 
	
	\medskip 
	
	
	
	\noindent \textit{Previous work}. There exist  few studies  on  coordinate descent methods when the second term in the objective function is nonseparable. For example,  \cite{Nec:13,NecTak:20,TseYun:09} considers  the composite optimization problem \eqref{eq:prob}  with $\psi$ convex and  separable (possibly nonsmooth) and $\phi$ the indicator function of the set  $\{x: Ax=b\}$.  Hence,   nonseparability comes from the linear constraints.  In these settings, \cite{Nec:13,NecTak:20,TseYun:09}  proposed coordinate proximal gradient descent methods that require solving at each iteration  a subproblem over a  subspace using  a part of the gradient of $f$ at the current feasible  point. For these algorithms  sublinear rates are derived in the (non)convex case and linear convergence is obtained for  strongly convex objective. Further, for  optimization problem    \eqref{eq:prob} with $\psi$ convex, nonseparable and nonsmooth,  \cite{LatThe:21} considered a proximal coordinate descent type method, where, however, at each iteration one needs to evaluate to full prox of $\psi$.  Note that since $\psi$ is nonsmooth, we can consider $\phi \equiv 0$.  Linear convergence results were  derived in \cite{LatThe:21}  when the objective function satisfies the Kurdyka-Lojasiewicz (KL) property.  Recently, \cite{AbeBec:21} considers the optimization problem \eqref{eq:prob},  with   $f$ quadratic  convex function  and $\psi$ convex function  (possibly nonseparable and nonsmooth). Hence, $\phi$ can be included in $\psi$.  In these settings,  \cite{AbeBec:21}  used  the forward-backward envelope to smooth the original problem and then solved the smooth approximation with an accelerated coordinate gradient descent method.   Sublinear rates were derived for this scheme, provided that the second function $\psi$ is Lipschitz continuous. However, in this method it is also necessary to compute the full prox of function $\psi$.  In  \cite{NecCho:21}, a stochastic coordinate  proximal gradient method is proposed, where  at each iteration one needs to sketch  the gradient   $\nabla f$  and compute the  prox of $\psi$  along some subspace  generated by a random matrix. Finally, in  \cite{NecCho:21} we consider  $\phi \equiv 0$. Assuming that $\psi$ is twice differentiable,  (sub)linear convergence rates are derived  for both convex and nonconvex settings. However, \cite{NecCho:21} requires the computation of a prox along some subspace, which still  can be prohibitive in some applications (including nonnegative matrix factorization).  
	
	\medskip 
	
	\noindent \textit{Contributions}.  This paper deals with large-scale composite  nonconvex optimization problems having  the  objective function formed as a sum of three  terms, see  \eqref{eq:prob}.   Under these  general settings, we present  a cyclic coordinate gradient  descent method  and derive convergence rates in function values.  More precisely, our   main contributions are:\\
	(i) We design a cyclic coordinate projected gradient method,  which requires at each iteration the computation of a block of components of the gradient of the differentiable part of the objective function. We propose a new stepsize strategy for this method,
	which guarantees descent and convergence under certain boundedness assumptions on hessian of  $\psi$. In particular, our stepsize rule is \emph{adaptive} and requires the computation of  a positive root of a given polynomial.\\
	(ii) We  prove that our algorithm is descent method and derive  convergence rates in function values. More precisely,  under Kurdyka-Lojasiewicz (KL) property, we derive sublinear, linear or superlinear rates  depending on the KL parameter.\\ 
	(iii) We show that the orthogonal nonnegative matrix factorization problem, see  \cite{ChiLu:19}, fits into our settings and thus we can solve it using our algorithm. Preliminary numerical results confirm the  efficiency of our method on this application.
	
	
	
	\medskip 
	
	\noindent \textit{Content}. The paper is organized as follows.    In Section II we present some definitions and preliminary results. In Section III we introduce our  coordinate projected gradient  algorithm and derive convergence rates in function values. Finally, in Section IV we provide detailed numerical simulations on the orthogonal nonnegative matrix factorization problem.


	\section{Preliminaries}
	\noindent  In this section we present our basic assumptions for composite problem  \eqref{eq:prob},   some definitions and  preliminary results.  We consider the following problem setting.  Let $U \in \mathbb{R}^{n\times n}$ be  a column permutation of the $n \times n$ identity matrix and further let $U = [U_{1},...,U_{N}]$ be a decomposition of $U$ into $N$ submatrices, with $U_{i} \in \mathbb{R}^{n \times n_{i}}$, where $\sum_{i = 0}^{N}n_{i} = n$. Any vector $x \in \mathbb{R}^n$ can be written uniquely as $x = \sum_{i = 0}^{N} U_{i} x^{(i)}$, where $x^{(i)} = U_{i}^{T}x \in \mathbb{R}^{n_{i}}$. 
	Throughout the paper the following assumptions will be valid:
	
	\begin{assumption} 
		\label{ass1}
		For composite optimization problem \eqref{eq:prob} the following assumptions hold:\\
		A.1: Gradient of $f$ is coordinate-wise Lipschitz continuous:
		\begin{align}
			\label{lip1}
			\|U_{i}^{T}(\nabla f(x + U_{i} h) - \nabla f(x)) \| \leq L_{i}\left(x^{\neq i}\right) \|h\|
		\end{align}
		\noindent for all $x \in \mathbb{R}^{n},  h \in \mathbb{R}^{n_{i}}$ and  $i = 1:N$. \\	 
		A.2:  Function $\psi$ is twice continuously differentiable (possibly nonseparable and nonconvex). Moreover, we assume that  there exist   integer $p\geq1$ and constants  $\Hpi > 0$ such that
		\begin{equation*}
			\| U_{i}^{T} \nabla^2 \psi (y) U_{i}\| \leq \Hpi \|y\|^{p}  \quad \forall y \in \mathbb{R}^{n} \quad i = 1:N.  
		\end{equation*} 
		A.3: Function $\phi $ is the indicator function of a nonempty closed convex separable set  $Q \subseteq  \mathbb{R}^{n}$, i.e.,
		\vspace*{-0.2cm}
		\begin{equation*}
			Q =  \Pi_{i = 1}^{N} Q_{i}, \quad \text{with} \quad Q_{i} \subseteq \mathbb{R}^{n_{i}}, \;\;  i=1:N.
			\vspace*{-0.2cm}
		\end{equation*}	
		A.4: A solution exists for \eqref{eq:prob}  (hence,   $F^* > -\infty$).
	\end{assumption}
	
	
	\noindent  In Assumption \ref{ass1}.[A1]  Lipschitz  constants $L_{i}\left(x^{\neq i}\right)$ may depend on $x^{(j)}$ for $j \neq i$.  
Further, let us  define:
	\vspace*{-0.2cm}
	\begin{equation}
		\label{eq:subg}
		S_F(x) = \text{dist}(0, \partial F(x)) \;  \left(:= \inf_{F_x \in \partial F(x)} \| F_x\| \right).\\
		\vspace*{-0.2cm}
	\end{equation}
	
	\noindent Let us define next Kurdyka-Lojasiewicz (KL) property  \cite{BolDan:07}. 
	
	
	\begin{definition}
		\label{def:kl}
		\noindent A proper and lower semicontinuous  function $F$ satisfies  \textit{Kurdyka-Lojasiewicz (KL)} property if for every compact set $\Omega \subseteq \text{dom} \, F$ on which $F$ takes a constant value $F_*$ there exist $q, \sigma_q, \delta, \epsilon >0$ such that   one has:
		\vspace*{-0.2cm}
		\begin{equation}
			\label{eq:kl}
			F(x) - F_*  \leq \sigma_q  S_{F}(x)^q, 
			\vspace*{-0.2cm}
		\end{equation}
		for all $x$ such that $\text{dist}(x, \Omega) \leq \delta, \; 
		F_* < F(x) < F_* + \epsilon$.
	\end{definition}     
	
	\medskip 
	
	\noindent  The basic idea of our algorithm consists of updating  $i$th component of  $x\in\mathbb{R}^{n}$   in a cyclic manner as $x^+ = x+ U_{i}d$,  for some appropriate direction  $d$.  Let us fix some notations first.  We denote $x_{k,0} = x_{k}$ and then $\xki = \left(x_{k+1}^{(1)},\cdots,x_{k+1}^{(i_{k})},x_{k}^{(i_{k}+1)}, \cdots, x_{k}^{(N)}\right)$. Further, let us define the differentiable function $h(x) = f(x) + \psi(x)$. For simplicity of the  exposition let us  denote $\Lik = L_{i_{k}}\left(\xkip^{\neq i_{k}}\right)$. Also, let  $\phi_{i}$ denote the indicator function of the individual  set $Q_i$ and thus $\phi (x) = \sum_{i=1}^N \phi_i(x^{(i)})$.


	\section{Coordinate projected gradient algorithm}
	\noindent In this section, we present a cyclic coordinate projected gradient descent algorithm for solving problem \eqref{eq:prob},  with $f$ and $\psi$ possibly nonseparable and nonconvex and $\phi$ the indicator function of a separable convex set.   Contrary to the usual approach from literature, the algorithm  disregards the composite form of the objective function and makes an update based on  the partial gradient  of the differentiable term  $\nabla h$ (recall  that in the literature, the convergence of  coordinate gradient descent methods is guaranteed only when $f$ is smooth and  $\psi$ is separable or $\psi \equiv 0$).  Hence, our cyclic Coordinate Projected  Gradient Descent (CPGD) algorithm~is: 
	\begin{algorithm}[H]	
		\caption{CPGD} \label{alg:RCPG}
		\begin{algorithmic} 	
			\STATE Given a starting point $x_{0} \in Q$
			\FOR {$k \geq 0$}
			\FOR {$i_{k} = 1:N$}
			\STATE 1. Choose $\Hf > \frac{\Lik}{2}$ and compute the positive root $\alphaki \geq 0$ of the following polynomial:
			\begin{align}
				&2^{p-1} \Hp \alpha^{p+1} + (2^{p-1}  \Hp \|\xkip\|^{p} + \Hf)\alpha \nonumber \\ &  = 
				\|U_{i_{k}}^{T} \nabla h(\xkip)\| . \label{eq:100}
			\end{align}
			\STATE 2. Update the stepsize in an adaptive fashion:
			\begin{equation*}
				\HF \!=\! 2^{p-1} \Hp \! \|\xkip\|^p + 2^{p-1} \Hp \! \alphaki^p + \Hf. 
			\end{equation*}
			\STATE 3. Solve the subproblem (projection):
			\begin{align}
				\label{eq:subproblem2}
				\dkik = \argmin_{d \in \mathbb{R}^{n_{i}}} & \langle U^{T}_{i_k} \nabla h(\xkip), d \rangle + \frac{\HF}{2} \|d\|^2   \nonumber\\
				&+ \phi_{i_{k}}(x_{k}^{(i_{k})} + d).
			\end{align}
			\STATE 4. Update $x_{k+1}^{(i_{k})} = x_{k}^{(i_{k})} + \dkik$.
			\ENDFOR
			\ENDFOR
		\end{algorithmic} 
	\end{algorithm}
 \vspace*{-0.2cm}
	\noindent The main difficulty with the algorithm CPGD is that we need to find an appropriate stepsize $\HF$ which ensures descent, although the differentiable part $h$ of the objective function  does not have a coordinate-wise Lipschitz gradient.  In the sequel we prove that   the  stepsize choice for $\HF$ from the algorithm CPGD combined with additional properties on $\psi$  yields descent. Note that \eqref{eq:subproblem2} can be written as:
	\begin{equation}
		\xbarik = \xkik - \dfrac{1}{\HF}  U^{T}_{i_k} \nabla h(\xkip) \label{eq:xbar}
	\end{equation}
	and
	\vspace*{-0.2cm}
	\begin{equation}
		\xkikn = \text{proj}_{Q_{i_{k}}}\left( \xbarik \right). \label{eq:proj}
	\end{equation}
	Moreover, from \eqref{eq:100}, we have:
	\begin{align*}
		&\left(2^{p-1}\Hp \alphaki^{p} + 2^{p-1}\Hp \|\xkip\|^{p} + \Hf \right) \alphaki\\
		& =  \|U_{i_{k}}^{T} \nabla h(\xkip)\|.
	\end{align*}
	\noindent Hence, using \eqref{eq:xbar}, we obtain:
	\begin{align}
		& \left\|\xbarik - \xkik  \right\| = \dfrac{1}{\HF}\|U_{i_{k}}^{T} \nabla h(\xkip)\| \nonumber \\
		&= \dfrac{\|U_{i_{k}}^{T} \nabla h(\xkip)\|}{2^{p-1}\Hp \alphaki^{p} + 2^{p-1}\Hp \|\xkip\|^{p} + \Hf}  = \alphaki. \label{eq:102}
	\end{align}
	
	\noindent Note that it is easy to show that the polynomial equation has a nonnegative root $\alpha_{k}$. 
	Let us define the following parameters:
	\vspace*{-0.2cm}
	\begin{equation}
		\Hf = \dfrac{\Lik + \eta_{i_{k}}}{2} \quad \text{and} \quad \eta_{\min} = \min_{k \in \mathbb{N} } \eta_{i_{k}}.  \label{eq:75}
	\end{equation}
	
	\begin{lemma}
		\label{lemma:1}
		If Assumption 1 holds, then the iterates of algorithm CPGD satisfy the following descent:
		\begin{align*}
			F(x_{k+1}) \leq F(x_{k}) - \dfrac{\eta_{\min}}{2} \|x_{k+1} - x_{k}\|^2. 
		\end{align*}
	\end{lemma}
	
	\begin{proof}
	See appendix for a proof.
	\end{proof}
	
	\medskip 
	
	\noindent  Next, we introduce some additional assumptions:
	
	\begin{assumption} 
	\label{ass2}
	Consider the following assumptions: \\
	A.5:  Sequence $(x_{k})_{k\geq 0}$ generated by  CPGD is  bounded and  there exists $L_{\max}$ such that $\Lik \leq L_{\max}$ for all $k \geq 0$.  \\
	A.6:  Assume that the full gradient  $\nabla f$ is Lipschitz continuous on any bounded subset of $\mathbb{R}^{n}$.
	\end{assumption}
	
	
	\noindent Let us define the following constant: 
	\vspace*{-0.2cm}
	\begin{equation}
		H_{F,\max} = \max_{ k \geq 0,i_{k}=1:N} \HF.  \label{eq:Hmax}
	\end{equation}
	Note that if Assumption 2.[A.5] holds, then $H_{F,\max}$ is finite. If  $(x_{k})_{k\geq 0}$ is bounded, from Assumption \ref{ass1}.[A.2]  there exists $\Hpsimax>0$ such that  $\Hpsimax = \max_{i=1:N, x \in \overline{\text{conv}}\{(x_{k})_{k\geq 0}\}} \|U^{T}_{i}\nabla^2 \psi (x)\| < \infty$. Moreover, in the next lemma we consider $L>0$ to be the  Lipschitz constant for the full gradient $\nabla f$ over $\text{conv}\{(x_{k})_{k\geq 0}\}$. 
Then, we have the following result. 
	
	
	\begin{lemma}
		\label{lemma:2}	
		Let $(x_{k})_{k\geq 0}$ be generated by  algorithm CPGD and $C= 4NL^2 + 4N\Hpsimax^2+ 2 H_{F,\max}$. If  Assumptions 1 and 2 hold, then:
		\vspace*{-0.2cm}
		\begin{equation}
			\left[S_{F}(x_{k+1})\right]^{2} \leq  C  \|x_{k+1} - x_{k}\|^2. 
		\end{equation}
	\end{lemma}
	
	\begin{proof} 
		See appendix for a proof. 
	\end{proof}
	
	\medskip 
	
	\noindent Let us denote the set of limit points of the sequence $(x_{k})_{k\geq 0}$ by $\Omega(x_{0})$.  Next lemma derives some  properties for   $\Omega(x_{0})$. 
	
	\medskip 
	
	\begin{lemma}
		\label{lemma:3}
		Let $(x_{k})_{k\geq 0}$ be generated by algorithm CPGD. If  Assumptions 1 and 2 hold, then $\Omega(x_{0})$ is  compact set,  $F$ is constant on $\Omega(x_{0})$ and $0 \in \partial F(\Omega(x_{0}))$.
	\end{lemma}
	
	\begin{proof}
	See appendix for a proof. 
	\end{proof}
	
	\medskip 
	
	\begin{lemma}
		\label{lemma:rec}
		Let $\{\Delta_{k}\}_{k\geq0}$ be a  sequence of positive numbers satisfying the following recurrence:  
		\vspace*{-0.2cm}
		\begin{equation}
			\Delta_{k} - \Delta_{k+1} \geq c \Delta_{k+1}^{\alpha+1} \quad \forall k \geq 0. 
			\vspace*{-0.2cm}
		\end{equation}
		\noindent Then,  we have:  \\
		(i) For $c=1$ and  $\alpha \in (0,1)$ the sequence  $\Delta_k$ converges to $0$ with sublinear rate: 
		\begin{equation}
			\Delta_{k} \leq \dfrac{ \Delta_{0}}{\left(  1 + \dfrac{\alpha k}{1 + \alpha} \ln(1 + \Delta_{0}^{\alpha})\right)^\frac{1}{\alpha}}. \label{eq:61}
		\end{equation} 	 
		(ii)  For  $c>0$ and $\alpha = 0$ the sequence  $\Delta_k$ converges to $0$ with linear rate: 
		\vspace*{-0.2cm}
		\begin{equation}
			\Delta_{k} \leq \left(\dfrac{1}{1+c}\right)^{k}\Delta_{0}.  
			\label{eq:62}
		\end{equation}	 
		(iii) For $c>0$ and $\alpha < 0$ the sequence  $\Delta_k$ converges to $0$ with superlinear rate: 
		\begin{equation}
			\Delta_{k+1} \leq \left(\dfrac{1}{1 + c \Delta_{k+1}^{\alpha}}\right) \Delta_{k}.  
			\label{eq:63}
		\end{equation} 	
	\end{lemma}
	
	\begin{proof}
	See appendix for a proof.
	\end{proof}
	
	\noindent In the next theorem we derive convergence rates for CPGD algorithm  when the objective function $F$ satisfies the KL condition. From simplicity of exposition, let us define:
	$$D = \dfrac{\eta_{\min}}{2\left( 4NL^2 + 4N\Hpsimax^2+ 2 H_{F,\max}\right) }.$$

	\begin{theorem}
		\label{theo:KL1}
		Let $(x_{k})_{k\geq 0}$ be generated by  algorithm CPGD. Let  Assumptions 1 and  2 hold and additionally, assume that  $F$ satisfy the KL property \eqref{eq:kl} on $\Omega(x_{0})$. Then:
		(i) If $q\in(1,2)$, we have the following sublinear rate:
		\begin{small}
			\begin{equation*}
				F(x_{k}) - F_{*} \leq \dfrac{F(x_{0}) - F_{*}}{\left( 1 \!+\! \frac{2-q}{2}k \ln \! \left(\! 1 \!+\! D\sigma_q^{-\frac{2}{q}} \! \left( F(x_{0}) - F_{*}\right)^{\frac{2-q}{q}}\right) \! \right)^\frac{q}{2-q}}
			\end{equation*}
		\end{small}
		(ii) If $q=2$, we have the following linear rate:
		\begin{equation*}
			F(x_{k}) - F_{*} \leq \left( \dfrac{\sigma_2}{\sigma_2 + D}\right)^{k} \left( F(x_{0}) - F_{*}\right). 
		\end{equation*}
		(ii) If $q>2$, we have the following superlinear rate:
		\begin{equation*}
			F(x_{k}) - F_{*} \leq  \dfrac{F(x_{k-1}) - F_{*}}{1 + D \sigma_{q}^{-\frac{2}{q}}\left(F(x_{k}) - F_{*}\right)^{\frac{2-q}{q}}}. 
		\end{equation*}
	\end{theorem}
	
	\begin{proof}
Proof 	follows from  Lemmas \ref{lemma:1}, \ref{lemma:2} and \eqref{eq:kl}. 
	\end{proof}


	\section{Orthogonal Nonnegative Matrix Factorization}
	\noindent We consider  a penalized formulation of the orthogonal nonnegative matrix factorization problem, used e.g., in dimensionality reduction of big data \cite{AhoHie:21,ChiLu:19}:
	\begin{equation} 
		\min_{W \in \mathbb{R}^{m\times r}_{+}, V \in \mathbb{R}^{r \times n}_{+} } \dfrac{1}{2} \| X - WV\|^{2}_{F} + \dfrac{ \lambda}{2} \|I-VV^{T}\|^2_{F}.  \label{eq:MF}
	\end{equation}
	
	\noindent Let us define the following functions: 
	\begin{align*}
		& f(W,V) =  \dfrac{1}{2} \| X - WV\|^{2}_{F}, \quad \psi(W,V) = \dfrac{ \lambda}{2} \|I-VV^{T}\|^2_{F}, \\
		& \phi(W,V) = \phi_{1}(W) + \phi_{2}(V),
	\end{align*} 
	where
	\begin{equation*}
		\phi_{1}(W) =  \left\lbrace\begin{array}{ll}  0
			\text{ if } W \in \mathbb{R}^{m\times r}_{+}, \\
			+ \infty \text{ otherwise} 
		\end{array}\right.
		\phi_{2}(V) =  \left\lbrace\begin{array}{ll}  0
			\text{ if } V \in \mathbb{R}^{r\times n}_{+}, \\
			+ \infty \text{ otherwise.} 
		\end{array}\right. 
	\end{equation*}
	
	\noindent For optimization problem \eqref{eq:MF} one can notice that  the functions $f(W,V)$ and $\psi(W,V)$ are polynomial and $\phi_{1}(W)$ and $\phi_{2}(V)$ are indicator functions of semialgebraic (polyhedral) sets,  thus 
	they are semialgebraic functions. Therefore, it follows from Theorem 3.1 in \cite{BolDan:07} that the objective function 
	$F(W, V) = f (W, V) +  \psi(W,V) + \phi(W,V)$ satisfies the  KL condition  for some  exponent $q>1$ (see \eqref{eq:kl}).
	
	\noindent  Note that we have the following expressions for the  gradient and the hessian of $f$: 
	\begin{align*}
		&\nabla_{W} f(W,V) = WVV^{T} \!-\! XV^{T}, \\    
		& \nabla_{V} f(W,V) = W^{T}WV - W^{T}X, \\
		& \nabla^{2}_{WW} f(W,V)Z = ZVV^{T}, \;\; \nabla^{2}_{VV} f(W,V)Z = W^{T}WZ.
	\end{align*}
	
	\noindent Thus $\nabla f$ is Lipschitz continuous w.r.t. $W$, with constant Lipschitz $L_{1}(V) = \|VV^{T}\|_{F}$.  Similarly,  $\nabla f$ is Lipschitz continuous w.r.t. $V$, with constant Lipschitz $L_{2}(W) = \|W^{T}W\|_{F}$.  This implies that $ \nabla f$ satisfies Assumption 1.[A.1].  Moreover, $\nabla_{V} \psi (W,V) = 2 \lambda (VV^{T}V - V)$ and the hessian is:
	\begin{equation*}
		\nabla^2_{VV} \psi (W,V)Z = 2 \lambda (ZV^{T}V + VZ^{T}V + VV^{T}Z - Z). 
	\end{equation*}
	\noindent Thus, we get the following bound on the hessian of $\psi$:
	\begin{align*}
		&\langle Z, \nabla^2_{V} \psi (W,V)Z \rangle \\
		&= \langle Z, 2 \lambda (ZV^{T}V + VZ^{T}V + VV^{T}Z - Z) \rangle \\ &\leq 6\lambda\|Z\|_{F}^{2} \|V\|_{F}^{2}. 
	\end{align*}
	
	\noindent This implies that $ \nabla^2_{V} \psi (\cdot)$ satisfies Assumption 1.[A.2], with $p=2$ and $H_{\psi_{2}} = 6\lambda$. Moreover, note that $H_{\psi_{1}} = 0$. Therefore, we can solve  problem \eqref{eq:MF} using algorithm CPGD,   having   the following updates: 
	\begin{align*}
		&W_{k+1} = \max\left( W_{k} - \dfrac{1}{\HfW(V_{k})}(W_{k}V_{k}V_{k}^{T} - XV_{k}^{T}),0\right), \\
		&\bar{V}_{k} = V_{k} - \dfrac{1}{\HF}\left( W_{k+1}^{T}W_{k+1}V_{k} - W_{k+1}^{T}X \right)  \\
		& \quad  - \dfrac{1}{\HF}\left(  2 \lambda (V_{k}V^{T}_{k}V_{k} - V_{k})\right), \;\; V_{k+1} = \max\left(\bar{V}_{k},0\right). 
	\end{align*}
	
	\noindent In the previous updates,  we choose $\HfW(V_{k}) > \dfrac{L_{1}(V_{k})}{2}$, $\HfV(W_{k+1}) > \dfrac{L_{2}(W_{k+1})}{2}$,   $\HF = 12 \lambda \|V_{k}\|_{F}^2 + 12\lambda \alphaki^2 + \HfV(W_{k+1})$ and $\alphaki$ as  the solution of  equation:
	\begin{align*}
		&12 \lambda \alpha^{3} + \left( 12\lambda  \|V_{k}\|_{F}^2  + \HfV(W_{k+1})\right) \alpha \\
		&- \|\nabla_{V} f(W_{k+1},V_{k}) +\nabla_{V} \psi(W_{k+1},V_{k}) \|_{F} = 0. 
	\end{align*}
	Note that one can find explicitly the positive root of this third order equation. Moreover, our algorithm is very simple to implement as it requires only basic matrix operations.   
	\begin{figure}[h]
		\includegraphics[height=0.2\textheight, width= 0.18\textheight]{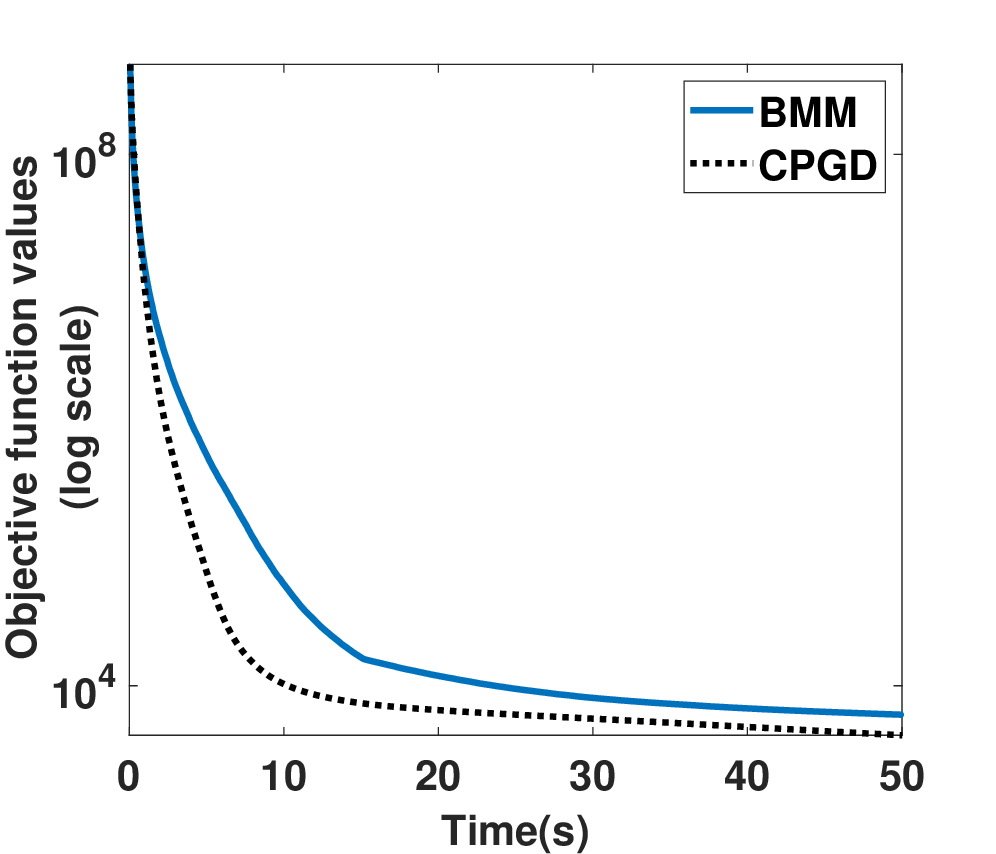} 
		\hspace{-0.4cm}
		\includegraphics[height=0.2\textheight, width= 0.18\textheight]{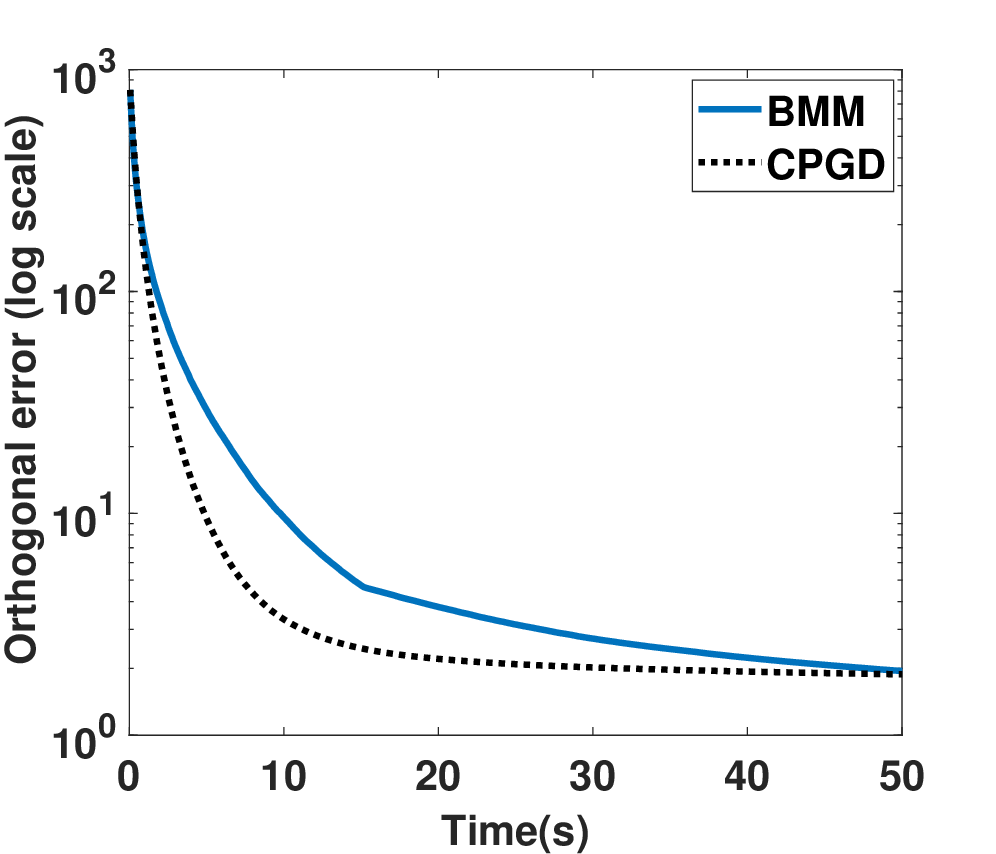}
		\caption{Salinas dataset for $r = 15$: evolution of objective function values (top figure) and orthogonal error (bottom figure) with respect to time of BMM and CPGD}
		\label{fig:1}
	\end{figure}

	\noindent In our experiments, we take $\HfW(V_{k}) =  0.51 \cdot L_{1}(V_{k})$ and $\HfV(W_{k+1}) = 0.51 \cdot L_{2}(W_{k+1})$. We compare our CPGD algorithm with the algorithm BMM proposed in \cite{HiePha21}. For problem \eqref{eq:MF}, BMM is a Bregman gradient descent method having computational cost per iteration comparable to CPGD. For numerical tests, we consider Salinas data set from \cite{DataSet}, which has the dimensions $111104 \times 204$. Each row of matrix $X$ is a vectorized image at a given band of the data set and we aim to reduce the number of bands using formulation \eqref{eq:MF}, i.e. to reduce $X \in \mathbb{R}^{111104 \times 204}$ to $W \in \mathbb{R}^{111104 \times r}$, with $r \ll 204$. We have $16$ classes for this image.  
	The matrices $W_0$ and $V_{0}$ are initialized random, but positive. Moreover, we take $\lambda = 1000$ and the dimension $r =15$. We run the two algorithms for $50$s. The results are displayed in Figures \ref{fig:1}, where we plot  the evolution of  function values (top) and the orthogonality error $O_{\text{error}} = \|I - V_{k}V_{k}^{T}\|_{F}$ (bottom) along time. From the plots, we observe that algorithm CPGD is better than BMM in terms of both,  function values  and  orthogonality error.

	\medskip
	\noindent We further test the quality of the reduced data $W$ for different ranks $r$ on a classification task. We choose the nonlinear Support Vector Machine (SVM) \cite{BosGuy:92} with the Radial Basis Function(RBF) for the kernel. We set the regularization parameter of the classifier to 10 and the kernel coefficient to $1/(r \cdot \text{Var}(W))$, where $\text{Var}(W)$ is the variation of $W$. Further, we divide randomly our data set in 80\% samples from each class for the training and the remaining 20\% for testing. Additionally, to reduce the influence of
	samples random selection, the classifier runs 10 times 	and we display the average results in Table \ref{tb:class}. The performance is quantified via the overall accuracy (OA) expressed in percentage, the Cohen's kappa coefficient ($\kappa$) \cite{GraBag:20} and the training time. Observe that for the reduced matrix data $W$ given by CPGD, we obtain better performances for the classifier than for the BMM   reduced matrix data,  in all the criteria (OA, $\kappa$ and time). Additionally,  one can observe that for $r =80$, we obtain for the classifier based on the  CPGD  reduced data matrix  even better overall accuracy than on the original data.
	 	
	\begin{table}[h]
		\caption{Classification results for Salinas data set obtained by applying  RBF SVM for  different CPGD vrs. BMM  (reduced) data matrices of rank  $r$.}
		\label{tb:class}
		\begin{tabular}{|c|c|c|c|c|c|c|}
			\hline
			{Alg.}&{r}& {5} & {15} & {50}& {80}& {204} \\
			\hline
			\parbox[t]{2mm}{\multirow{3}{*}{\rotatebox{90}{CPGD}}} & {OA(\%)} &{90.52}&{91.17}&{91.63}&{92.26}&{92.05}\\	
			\cline{2-7}
			{}&	{$\kappa$} &{0.8942}&{0.9014}&{0.9064}&{0.9136}&{0.9112}\\
			\cline{2-7}
			{}	&{time(s)}& {7.02} & { 8.19}& {12.97}& {22.04}& {31.3}\\
			\hline\hline
			\parbox[t]{2mm}{\multirow{3}{*}{\rotatebox{90}{BMM}}}&{OA(\%)} &{90.43} & {91.16}& {91.42}& {91.83}& {92.05}\\	
			\cline{2-7}
			{}&{$\kappa$} &{0.8931} & {0.9013}& {0.9041}& {0.9088}& {0.9112}\\
			\cline{2-7}
			{}&{time(s)}& {7.37} & {13.5}& {15.70}& {26.66}& {31.3}\\
			\hline
		\end{tabular}
	\end{table}

	\section{CONCLUSIONS}
	\noindent In this paper we have considered composite problems having  the  objective function formed as a sum of three  terms, first has block coordinate-wise Lipschitz continuous gradient, second  is twice differentiable but nonseparable and third is the indicator function of some separable closed convex set. We have proposed  a new cyclic coordinate projected gradient descent method for solving this problem. Moreover, we have designed a new adaptive stepsize strategy. Further, we derived  convergence bounds in the  nonconvex settings.  Numerical results confirmed the  efficiency of our  algorithm on the orthogonal nonnegative matrix factorization problem.


	\section*{ACKNOWLEDGMENT}
	\noindent The research leading to these results has received funding from: ITN-ETN project TraDE-OPT funded by the European Union’s Horizon 2020 Research and Innovation Programme under the Marie Skłodowska-Curie grant agreement No.  861137; NO Grants 2014-2021, under project ELO-Hyp,  no. 24/2020;  UEFISCDI PN-III-P4-PCE-2021-0720, under project L2O-MOC, nr.  70/2022.


\appendix
\textit{Proof of Lemma \ref{lemma:1}:} 	Since $\psi$ is differentiable, from mean value theorem, there exists $y_k \in [\xkip,\xki]$ such that
$\psi(\xki) - \psi(\xkip) = \langle \nabla \psi(y_{k}), \xki - \xkip \rangle = \left\langle \nabla \psi(y_{k}), U_{i_{k}}\left( x_{k+1}^{(i_{k})} - x_{k}^{(i_{k})} \right)  \right\rangle$.
Combining the last equality with Assumption 1.[A.1] and Step 4 in Algorithm 1, we get:
\vspace*{-0.1cm}
\begin{align}
	&f(\xki)+\psi(\xki) \leq
	f(\xkip) + \psi(\xkip) \label{eq:1} \\  &+ \frac{\Lik}{2} \|\dkik\|^2  
	 + \left\langle U_{i_{k}}^{T}\nabla \psi(y_{k}) + U_{i_{k}}^{T}\nabla f(\xkip),  \dkik\right\rangle \nonumber 
	\vspace*{-0.1cm}
\end{align}	

\noindent On the other hand, from the optimality condition for the problem \eqref{eq:subproblem2}, we have: 
\vspace*{-0.1cm}
\begin{equation*}
	- U_{i_{k}}^{T} \nabla h(\xkip) - \HF \dkik \in \partial \phi_{i_{k}}\left(x_{k+1}^{(i_{k})}\right). 
	\vspace*{-0.1cm}
\end{equation*}
\noindent Moreover, since $\phi$ is convex function, we get:
\vspace*{-0.1cm}
\begin{equation*}
	\phi(\xki) \leq \phi(\xkip) - \langle U_{i_{k}}^{T} \nabla h(\xkip) + \HF \dkik, \dkik \rangle.
	\vspace*{-0.1cm}
\end{equation*}
\noindent Since $\nabla h(x) = \nabla f(x) + \nabla \psi(x)$, combining the last inequality with \eqref{eq:1}, we obtain:	
\vspace*{-0.2cm}
\begin{align}
	& F(\xki) \leq F(\xkip) + \frac{\Lik}{2} \|\dkik\|^2   \\ 
	& + \langle U_{i_{k}}^{T}\left(  \nabla \psi(y_{k}) - \nabla \psi(\xkip)\right) , \dkik\rangle - \HF \|\dkik\|^2. \nonumber \label{eq:101} 
	\vspace*{-0.2cm}
\end{align}

\noindent From mean value inequlity there exists also  some $\bar{x}_k \in [\xkip,y_{k}]$ such that 
$U^{T}_{i_k}(\nabla \psi(y_{k}) - \nabla \psi(\xkip))  \leq \|U^{T}_{i_k}\nabla^2 \psi (\bar{x}_k) U_{i_{k}}\| \| y^{(i_{k})}_{k} - x_{k}^{(i_{k})}\| $.
Note that $\bar{x}_k =  (1-\mu) \xkip + \mu y_{k} $ for some $\mu \in [0,1]$. From  Assumption 1.[A.2] and last inequality we obtain:
\vspace*{-0.1cm}
\begin{align*}
	&\langle U_{i_{k}}^{T}\left(  \nabla \psi(y_{k}) - \nabla \psi(\xkip)\right) , \dkik\rangle \\ 
	&\leq \left\|U^{T}_{i_k}\nabla^2 \psi (\bar{x}_k) U_{i_k}\right\| \left\|y^{(i_{k})}_{k} - x_{k}^{(i_{k})}\right\| \left\|\dkik\right\| \\
	&\leq \Hp \|\xkip \!+\! \mu \left(y_{k} \!-\! \xkip\right) \|^{p} \left\|y_{k} \!-\! \xkip\right\|   \|\dkik\|.
	\vspace*{-0.1cm}
\end{align*}

\noindent Since $y_k \in [\xkip,\xkip+U_{i_k}\dkik]$, then  $\|y_{k}-\xkip\| \leq \|\dkik\|$. Moreover, since $\|a + b\|^p \leq 2^{p-1} \|a\|^p + 2^{p-1} \|b\|^p$ for any $p \geq 1$, we further get:
\vspace*{-0.1cm}
\begin{align*}
	& \langle U_{i_{k}}^{T}\left(\nabla \psi(y_{k}) - \nabla \psi(x_{k})\right) , d_{k}\rangle \\ 
	&\leq 2^{p-1} \Hp \left( \|\xkip\|^{p} +  \mu^p \|y_{k}-\xkip\|^{p}\right) \|\dkik\|^2 \\
	&\leq 2^{p-1} \Hp \|\xkip\|^{p}\|\dkik\|^2 + 2^{p-1} \Hp \|\dkik\|^{p+2}.
	\vspace*{-0.1cm}
\end{align*}

\noindent From inequality above and the expression of $\HF$, we get:
\vspace*{-0.1cm}
\begin{align*}
	&F(\xki) \\
	& \leq F(\xkip) + 2^{p-1} \Hp \|\dkik\|^{p+2} 
	- \HF \|\dkik\|^2 \\
	&+2^{p-1} \Hp \|\xkip\|^{p}\|\dkik\|^2 + \frac{\Lik}{2} \|\dkik\|^2. 
	\vspace*{-0.1cm}
\end{align*}
\noindent On the other hand, using \eqref{eq:xbar} and \eqref{eq:proj}, we have:
\begin{equation*}
	\left\|\dkik\right\| = \left\|x_{k+1}^{(i_{k})} - x_{k}^{(i_{k})}\right\| \leq  \left\|\xbarik - \xkik  \right\|.
\end{equation*}
\noindent Hence,
\vspace*{-0.1cm}
\begin{align*}
	&F(\xki) \leq F(\xkip) + \left( \frac{\Lik}{2} - \HF\right)  \|\dkik\|^2\\
	&\!+\!\! \left( \! 2^{p-1} \! \Hp  \! \|\xkip\|^{p} \!
	+ \! 2^{p-1} \! \Hp \!\! \left\|\xbarik  \!-\! \xkik  \right\|^p \right) \!\! \|\dkik\!\|^{2}  
	\vspace*{-0.1cm}
\end{align*}

\noindent From \eqref{eq:102}, we have $\HF = 2^{p-1} \Hp\|x_{k}\|^{p} + 2^{p-1} \Hp \left\|\xbarik - \xkik  \right\|^p + \Hf$. Using  \eqref{eq:75}, we obtain:
\vspace*{-0.1cm}
\begin{align*}
	F(\xki) &\leq F(\xkip) - \left( \Hf - \frac{\Lik}{2}\right)  \|\dkik\|^2 \\ 
	&\leq F(\xkip) - \frac{\eta_{\min}}{2} \|x_{k+1}^{(i_{k})} - x_{k}^{(i_{k})}\|^2.
	\vspace*{-0.1cm}
\end{align*}
\noindent Summing this  for $i_{k} = 1:N$, the lemma is proved. \\

\textit{Proof of Lemma \ref{lemma:2}:} Since $h$ is assumed differentiable and $\phi$ is convex, then from basic subdifferential calculus rules  it follows that $\nabla h(x) + \partial \phi(x) = \partial F(x)$ for any $x \in Q$, see  \cite{RocWet:98}. Moreover, from the optimality condition of \eqref{eq:subproblem2}, there exists $\phi_{x_{k+1}^{(i_{k})}} \in \partial \phi_{i_{k}}\left(x_{k+1}^{(i_{k})}\right)$ such that: \vspace*{-0.1cm} 
\begin{equation}
	U_{i_{k}}^{T} \nabla h(\xkip) + \HF \dkik + \phi_{x_{k+1}^{(i_{k})}} = 0. \label{eq:opt}
	\vspace*{-0.1cm}
\end{equation}

\noindent On the other hand,  $F_{x_{k+1}} := \nabla h(x_{k+1}) + \phi_{x_{k+1}}  \in \partial F(x_{k+1})$. Since $\phi$ is block separable function, it follows that $\phi_{x_{k+1}} = \sum_{i_{k}}^{N}U_{i_{k}}\phi_{x_{k+1}^{(i_{k})}}$. Using \eqref{eq:opt}, we get:
\vspace*{-0.1cm}
\begin{align*}
	&\left[S_{F}(x_{k+1})\right]^{2} \leq \|F_{x_{k+1}}\|^{2} = \|\nabla h(x_{k+1}) + \phi_{x_{k+1}} \|^2 \\
	&= \sum_{i_{k}=1}^{N}  \left\|U_{i_{k}}^{T}\nabla h(x_{k+1}) - U_{i_{k}}^{T} \nabla h(\xkip) - \HF \dkik\right\|^2 \\
	&\leq \sum_{i_{k}=1}^{N} 2\|U_{i_{k}}^{T}\nabla h(x_{k+1}) - U_{i_{k}}^{T} \nabla h(\xkip)\|^2 \\
	& + \sum_{i_{k}=1}^{N} 2\HF\|\dkik\|^2,
	\vspace*{-0.1cm}
\end{align*}

\noindent where we used  $\|a + b\|^2 \leq 2 \|a\|^2 + 2 \|b\|^2$ in the last inequality. From $\nabla h(x) = \nabla f(x) + \nabla \psi(x)$ and \eqref{eq:Hmax}, we have:
\vspace*{-0.1cm}
\begin{align}
	\label{eq:107}
	&\left[S_{F}(x_{k+1})\right]^{2} \leq \|F_{x_{k+1}}\|^{2} \nonumber\\ 
	&\leq \sum_{i_{k}=1}^{N} 4\|U_{i_{k}}^{T}\nabla f(x_{k+1}) - U_{i_{k}}^{T} \nabla f(\xkip)\|^2  \\
	& \!+ \! \sum_{i_{k}=1}^{N} \!\! 4\|U_{i_{k}}^{T}\nabla \psi(x_{k+1}) -\! U_{i_{k}}^{T} \nabla \psi(\xkip)\|^2  \!+\! 2 H_{F,\max} \|d_k\|^2. \nonumber 
	\vspace*{-0.1cm}
\end{align}

\noindent Since  $\nabla f$ is  Lipschitz over $\text{conv}\{(x_{k})_{k\geq 0}\}$ (see Assumption \ref{ass2}.[A.6]), then there exists $L>0$ such that
\vspace*{-0.1cm}
\begin{align}
	& \|U_{i_{k}}^{T}\nabla f(x_{k+1}) - U_{i_{k}}^{T} \nabla f(\xkip)\|   \leq L \|x_{k+1} - \xkip\|. 
	\label{eq:108}
	\vspace*{-0.1cm}
\end{align}

\noindent Since $\psi$ is twice differentiable, then from the mean value inequality there exists $\bar{x}_{i_k} \in [\xkip,x_{k+1}]$ such that:
\vspace*{-0.1cm}
\begin{align}
	&\|U^{T}_{i_k}(\nabla \psi(x_{k+1}) - \nabla \psi(\xkip))\| \nonumber\\
	 &\leq \|U^{T}_{i_k}\nabla^2 \psi (\bar{x}_k)\| \|x_{k+1} - \xkip\|. \label{eq:109}
	\vspace*{-0.1cm}
\end{align}

\noindent Since $(x_{k})_{k\geq 0}$ is assumed bounded,  then  $\overline{\text{conv}}\{(x_{k})_{k\geq 0}\}$ is also bounded. Using the fact that $\psi$ is twice continuously  differentiable, we have that there exist $\Hpsimax < \infty$ such that $\|U^{T}_{i_k}\nabla^2 \psi (x) \| \leq \Hpsimax$ for all $x \in \overline{\text{conv}}\{(x_{k})_{k\geq 0}\}$ and  $i_{k} = 1:N$. Note $[\xkip,x_{k+1}] \subseteq [x_{k},x_{k+1}]$, hence $\bar{x}_{i_k} \in \overline{\text{conv}}\{(x_{k})_{k\geq 0}\}$. From \eqref{eq:107}, \eqref{eq:108} and \eqref{eq:109}, we get:
\vspace*{-0.1cm}
\begin{align}
	&\left[S_{F}(x_{k+1})\right]^{2} \leq \|F_{x_{k+1}}\|^{2} \nonumber\\
	&\leq  \left(4NL^2 + 4N\Hpsimax^2 + 2 H_{F,\max}\right)  \|x_{k+1} - x_{k}\|^2,  \label{eq:6}
	\vspace*{-0.1cm}
\end{align}
where we used that $\|\xkip-x_{k+1}\| \leq \|x_{k+1} - x_{k}\|$ and $x_{k+1} = x_{k} + d_{k}$. This proves our statement.\\

\textit{Proof of Lemma \ref{lemma:3}:} 	Since the  sequence $(x_{k})_{k\geq 0}$ is bounded, this implies that the set $\Omega(x_{0})$ is also bounded. Closeness of $\Omega(x_{0})$ also follows observing that $\Omega(x_{0})$ can be viewed as an intersection of closed sets, i.e., $\Omega(x_{0}) = \cap_{j \geq 0} \cup_{\ell \geq j} \{x_{k}\}$. Hence $\Omega(x_{0})$  is a compact set. Let us prove that $F$ is constant on $\Omega(x_{0})$.  From Lemma \ref{lemma:1}, we have: 
\vspace*{-0.1cm}
\begin{equation*}
	\sum_{j=0}^{k} \|x_{j+1} - x_{j}\|^2 \leq \dfrac{2}{\eta_{\min}}\left( F(x_{0}) - F^{*}\right)  < \infty.
	\vspace*{-0.1cm}
\end{equation*}
\noindent This implies that $\|x_{j+1} - x_{j}\| \to 0$ as $j\to \infty$. Hence, from \eqref{eq:6}, we have:
\vspace*{-0.1cm}
\begin{equation}
	\lim_{j \to \infty} \|F_{x_{j+1}}\| = 0. \label{eq:58}
	\vspace*{-0.1cm}
\end{equation}

\noindent Moreover, from Lemma \ref{lemma:1} we have that $(F(x_{k}))_{k\geq 0}$ is monotonically decreasing and since $F$ is assumed bounded from below by $F^{*} > - \infty$, it converges, let us say to $F_{*} > - \infty$, i.e. $F(x_{k}) \to F_{*}$ as $k \to \infty$, and   $F_{*}  \geq  F^{*}$.
On the other hand, let $x_{*}$ be a limit point of $(x_{k})_{k\geq 0}$, i.e. $x_{*} \in \Omega(x_{0})$. This means that there is a subsequence $(x_{\bar{k}})_{\bar{k}\geq 0}$ of $(x_{k})_{k\geq 0}$
such that $x_{\bar{k}} \to x_{*}$ as $\bar{k} \to \infty$. 
From the lower semicontinuity of $\phi$ we always have:
	$\lim_{\bar{k} \to \infty} \inf \phi(x_{\bar{k}}) \geq \phi(x_{*})$. 
On the other hand, since $\phi$ is convex, we have: 
\begin{equation*}
	\phi (x_{\bar{k}}) \leq \phi (x_{*}) + \langle \phi_{x_{\bar{k}}}, x_{\bar{k}} - x_{*} \rangle \leq \phi (x_{*}) + \|\phi_{x_{\bar{k}}}\| \|x_{\bar{k}} - x_{*}\|. 
\end{equation*}
Moreover, from \eqref{eq:58} we have that $(\phi_{x_{\bar{k}}})_{\bar{k} \geq 0}$ is bounded. Since $x_{\bar{k}} \to x_{*}$, we get:
\begin{equation*}
	\lim_{\bar{k} \to \infty} \sup \phi(x_{\bar{k}}) \leq \lim_{\bar{k} \to \infty} \sup \phi (x_{*}) + \|\phi_{x_{\bar{k}}}\| \|x_{\bar{k}} - x_{*}\| \leq \phi(x_{*}).
\end{equation*} 
\noindent Therefore, we have  $\phi(x_{\bar{k}}) \to \phi(x_{*})$ and since  $f$ and $\psi$ are continuous functions, it also follows that  $F(x_{\bar{k}}) \to F(x_{*})$. This implies that $F(x_{*}) = F_{*}$. Finally,  to prove that $0 \in \partial F(\Omega(x_{0}))$, one can note that when $\bar{k} \to \infty$, we have $x_{\bar{k}} \to x_{*}$ and $F(x_{\bar{k}}) \to F(x_{*})$. Moreover, from \eqref{eq:58}, $\|F_{x_{\bar{k}}}\| \to 0$. Then, from the definition of the limiting subdifferential it follows that $0 \in \partial F(x_{*})$. 	\\

\textit{Proof of Lemma \ref{lemma:rec}:} 	First, the inequality \eqref{eq:61} has been derived in  \cite{Nes:19} (Lemma 11). If $\alpha = 0$, we have:
\begin{align*}
	& \Delta_{k} - \Delta_{k+1} \geq c  \Delta_{k+1} \iff 
	\Delta_{k} \geq \left( 1 + c  \right) \Delta_{k +1} \nonumber  \\ 
	& \iff \Delta_{k+1} \leq \left(\frac{1}{1 + c }\right) \Delta_{k}.
\end{align*}	
\noindent This yields \eqref{eq:62}.  Finally, for the third  case, we have: 
\begin{equation*}
	\Delta_{k+1}\left(1 + c \Delta_{k+1}^{\alpha}\right) \leq \Delta_{k} \!\iff\! \Delta_{k+1} \leq \left(\dfrac{1}{1 + c \Delta_{k+1}^{\alpha}}\right) \Delta_{k}. 
\end{equation*}
These prove our statements. 	\\   



	
\end{document}